\documentclass[11pt]{amsart} 
\usepackage{amsmath,amssymb}

\renewcommand{\epsilon}{\varepsilon}
\renewcommand{\phi}{\varphi}

\newcommand{\su}{\subseteq}
\newcommand{\rest}{\restriction}
\renewcommand{\a}{\alpha}
\renewcommand{\b}{\beta}
\newcommand{\g}{\gamma}

\newcommand{\ku}{\kappa}
\renewcommand{\d}{\delta}
\renewcommand{\l}{\lambda}
\renewcommand{\k}{\kappa}
\newcommand{\ka}{\kappa}
\newcommand{\xig}{\Upsilon} 
\newcommand{\miq}{\partial} 

\newcommand{\D}{\partial}

\newcommand{\PP}{\mathbb P}

\newcommand{\z}{\zeta}

\newcommand{\om}{\omega}

\newcommand{\lng}{\langle}
\newcommand{\rng}{\rangle}
\newcommand{\ov}{\overline}
\newcommand{\sm}{\setminus}

\newcommand{\Li}{{\operatorname {\chi_\ell} }}

\newcommand{\col}{{\operatorname {Col}}}
\newcommand{\dom}{{\operatorname {dom}\,}}
\newcommand{\cf}{{\operatorname {cf}}}

\newcommand{\otp}{{\operatorname {otp}}}
\newcommand{\ran}{{\operatorname  {Rang}}}
\newcommand{\stt}{{\rm st}}

\newcommand{\pcf}{\operatorname {pcf}}

\newcommand{\Reg}{\rm {Reg}}

\newcommand{\force}{\Vdash}
\newcommand{\imply}{\Rightarrow}

\newtheorem{theorem}{Theorem}[section]
\newtheorem{conclusion}[theorem]{Conclusion}
\newtheorem{corollary}[theorem]{Corollary}
\newtheorem{lemma}[theorem]{Lemma}

\newtheorem{definition}[theorem]{Definition}
\newtheorem{claim}[theorem]{Claim}

\author{Saharon Shelah}

\thanks{The author thanks the Israel Science Foundation for partial
support of this research, Grant no. 1053/11. Publication 1052}

\address{Department of Mathematics\\
Hebrew University of Jerusalem\\
Jerusalem\\
9190401  
 Israel}

\title[Lower Bounds]{Lower bounds on coloring numbers from hardness hypotheses in pcf theory}

\usepackage[hidelinks]{hyperref}
\renewcommand{\a}{\alpha}
\renewcommand{\b}{\beta}
\renewcommand{\g}{\gamma}

\renewcommand{\d}{\delta}
\renewcommand{\l}{\lambda}
\renewcommand{\k}{\kappa}
\renewcommand{\D}{\partial}
\renewcommand{\z}{\zeta}

\begin{document}
\makeatletter\def\shfiuwefootnote{\gdef\@thefnmark{}\@footnotetext}\makeatother\shfiuwefootnote{Version 2016-02-18\_11. See \url{https://shelah.logic.at/papers/1052/} for possible updates.}
\maketitle

\begin{abstract}
We prove that the statement ``for every infinite cardinal $\ku$, every graph with list-chromatic number 
$\ku$ has coloring number at most $\beth_\om(\ku)$" proved by Kojman \cite{koj}  using 
 the RGCH theorem 
\cite{sh:460} implies the WRGCH theorem, which is a weaker relative of the RGCH,  via a short forcing argument. 

Similarly, 
 a better upper bound than $\beth_\om(\ku)$ in this 
statement  implies  stronger (consistent) forms of the WRGCH theorem, 
  the consistency of whose negations  is wide open. 

Thus, the optimality of  Kojman's upper bound is 
 a purely  cardinal arithmetic problem,  and, as discussed below,
 is  hard to decide.  
\end{abstract}

\section{Introduction}

Recall that  the \emph{list-chromatic} or \emph{choosability} 
number of a graph $G=\lng V, E\rng$ is $\k$ if $\k$ is the least cardinal
such that for any assignment of lists of colors $L(v)$ to all vertices
$v\in V$ such that $|L(v)|\ge \k$ there exists a proper vertex coloring $c$ of $G$ with colors
from the lists, namely $c(v)\in L(v)$ for all $v\in V$. A graph $G$ has \emph{coloring number} $\k$ if $\k$ is the least cardinal such that 
 there exists a well-ordering $\prec$ on $V$ such that a vertex 
$v\in V$ is joined by edges to only $<\k$ vertices $u$ satisfying $u\prec v$.

Alon \cite{A} proved that  every finite graph with list-chromatic number $n$ has coloring number at most  $(4+o(1))^n$ and this bound is tight up to a factor of $2+o(1)$ 
by \cite{ERT}. 

In \cite{koj} Kojman used the Revised GCH theorem from  cardinal arithmetic
 \cite{sh:460} to prove in  ZFC the upper bound of  
   $\beth_\om(\ku)$  on the
coloring number of any  graph with a  list chromatic number $\le
\ku$, where $\beth_\om(\ku)$ is the cardinal gotten by applying the exponent function to $\ku$ infinitely many times.\footnote{Formally, $\beth_0(\ku)=\ku$, $\beth_{n+1}(\ku)=2^{\beth_n(\ku)}$ and $\beth_\om(\ku)=\lim_n\beth_n(\ku)$.}

By Erd\H os and Hajnal \cite{eh} from 1966,  
if the GCH is assumed,
 $\ku^{++}= (2^\ku)^+=(\beth_1(\ku))^+$ 
bounds the coloring number of every graph with 
 list-chromatic number $\ku$ for every  infinite $\ku$.
It is now known 
that  much weaker axioms than the GCH --- certain weak consequences of the Singular Cardinals
Hypothesis --- imply the same  upper bound  (see the second section in \cite{koj}), 
so in ``many" models of set theory, the upper bound is $(2^{\ku})^+$. 
 Komjath \cite{kom} recently improved the  GCH  upper bound to $2^\ku=\ku^+$, 
 constructed models of the GCH  in which $\Li(G)=\col(G)$ for every graph with infinite $\Li(G)$ and showed that in MA models $2^\ku$ is required. 

The sharp contrast between the  single exponent 
  in the bound  for the finite case, 
or  in the bound for the 
infinite case in the presence of  
 mild cardinal arithmetic  axioms,  and the
 infinite tower of exponents  in  the  ZFC bound, 
 led Kojman 
 to ask whether the upper bound  $\beth_\om(\ku)$ could be lowered in ZFC. He also asked   whether 
 the use of the RGCH  in proving his $\beth_\omega$  bound was necessary. 

 We prove here  that: 
 \begin{enumerate}
 	\item
 Kojman's $\beth_\omega$
upper bound implies the so called  Weak Revised GCH theorem (WRGCH) in pcf theory.
   \item
  better upper bounds  imply open 
strengthenings of the WRGCH theorem.
 \end{enumerate}
 
All implications above are via  standard forcing arguments. 
 
 Thus, improving Kojman's 
 upper bound on the coloring number of a graph in terms of its list-chromatic number  will be at least as hard as improving the WRGCH theorem.  In particular, 
a  better upper bound  cannot be 
gotten with only  graph-theoretic 
arguments.

Note that consistency results in pcf are hard; 
only recently Gitik \cite{Git15} succeed to make a 
remarkable  advance: for a countable set
$ {\mathfrak a} $ of regular cardinals,
$ {\rm pcf}( {\mathfrak a} ) $ 
may be  uncountable.
Grand as it is, this is a far cry from what is needed to show that the WRGCH cannot be improved. If, however,   all 
relevant strengthenings of the WRGCH  are indeed not provable in ZFC,
 then Kojman's $\beth_\om$ bound is 
optimal (a more detailed discussion of this is given below).

 \medskip
 
 We thank Menachem Kojman for asking us the
 questions and for reading two earlier versions of this proof and the present one. 
 
\subsection{Description of the reduction}  
  
\begin{definition} \label{wrgch} The Weak Revised GCH theorem,
	 WRGCH, is the statement that for 
every strong limit cardinal $\mu>\aleph_0$, 
e.g. $\beth_\omega$, and $\lambda > \mu$, for 
some $\kappa<\mu$ there is no  sequence 
$\ov \l = \lng \ov\l_i:i<\mu\rng$ of finite sequences of regular cardinals in $(\mu,\l)$ such that $J_{<\l}(\ov \l)\su [\mu]^{<\k}$. 
\end{definition}

 Here, the $\pcf$ operator is extended to sets  $\mathfrak a$ of finite sets (as above, we identify a finite sequence of cardinals with its
range) 
by
letting $\pcf(\mathfrak a)$ be interpreted as   $\pcf(\bigcup \mathfrak a)$ and similarly $J_{\l}(\mathfrak a)$ stands for $\{\mathfrak b: \mathfrak b\su \mathfrak a \,\& \, \max\pcf(\mathfrak b)<\l\}$.

The WRGCH is a straightforward consequence of the
 Revised GCH theorem \cite{sh:460}. 

Now  consider, for a natural number $m
\le  
1$,  
the following two closely related  statements 
(even equivalent, see below) 
with parameter $m$,
which are stronger than the WRGCH:
  
 \begin{itemize}
 \item[$\oplus^1_m$] there are no cardinal $\ku$ and 
   set $\mathfrak a$ of $\beth_m(\ku)$ many finite sequences of 
  regular cardinals, each larger than 
 $\beth_m(\ku)$, such that  $J_{<\sup ( \mathfrak 
    a)}[\mathfrak a] \su [\mathfrak a]^{<\ku}$,  
i.e.  $\mathfrak b\in [\mathfrak a]^\ku$ 
    implies that $\max\pcf (\mathfrak b)\ge
 \sup (\mathfrak a ) $.  
\end{itemize}

\begin{itemize}
 \item[$\oplus^2_m$] there are  no cardinals 
$\ku$ and 
 $ \xig  $ 
satisfying 
$ \miq:=\beth_m(\ku)\le  
\xig 
$ and a family of sets 
$\mathcal A\su 
[\xig]^ \miq $ 
 such that  
$|A\cap B|<\ku$ for all distinct
$A,B\in \mathcal A$ and $|\mathcal A|>
\xig $ 	

\end{itemize} 

The status of the statements above is as follows. If $m<n$ then
$\oplus^i_m$ implies $\oplus^i_n$ for $i\in \{1,2\}$.
All $\oplus^i_m$ hold in models of the GCH or even of 
just the strong hypothesis (see \cite{sh:420}, \S 6),  so are consistent with the 
axioms of ZFC. 

The question  for which $m$ is  $\oplus^i_m$ a theorem of
 ZFC is wide open, that is, for all $m\ge 1$, neither a ZFC proof\
nor a consistency of the negation is known at 
the moment. 
 The WRGCH, however, is a theorem of ZFC, as it follows 
trivially from the RGCH theorem. 

\medskip 

Lowering  $\beth_\om(\ku)$ in Kojman's upper bound to 
 $\beth_n(\ku)$ for
some $n<\om$   is at least as 
hard as proving the equivalent  
statements $\oplus^1_m$ and  $\oplus^2_m$  for   $m=2n+1$. 
The reason for this is that
 if the configuration that is 
  forbidden by e.g. $\oplus^2_m$ does exists in some  model $\mathbf V$ of ZFC then in some forcing extension of $\mathbf V$ there is a graph with list-chromatic number $\theta$ and coloring number $>\beth_n(\theta)$, for some $\theta>\ku$. 
The relation  $m= 2n+1$ can probably be tightened, but we made no effort to do so.

  Also, if it is assumed to the contrary that the configuration that is 
  forbidden by the WRGCH does exists
 in some universe $\mathbf V$ of set-theory,  then  
  in some forcing  extension $\mathbf V^{\mathbb P}$ of $\mathbf V$ there is 
   a graph  with  list-chromatic number   $<\ku$ and coloring number $> \beth_\om(\ku)$,  
   contrary to the $\beth_\omega$ upper bound. Thus, 
     the $\beth_\om$ graph-theoretic bound implies the WRGCH.

\bigskip


   We discuss next the  pcf-theoretic statements and  
   explain further their connection to upper bounds on coloring numbers.

\medskip
%
%
%
%

Let $\k\le 
\miq < \xig
 <\l=\cf(\l)$ be cardinals. 
Consider the statement:

\begin{itemize}
\item[$(st)^1_{\k,
\miq,\xig	
,\lambda}$] there is a $\mathcal A\su
[\xig]^ \miq $ 
of cardinality $\l$ such that if $A_1\not=A_2$ belong to
  $\mathcal A$ then $|A_1\cap A_2|<\k$.
\end{itemize}

We agree that if $\lambda=
\xig^+ $ 	
 we may omit it and if $\mu=
\xig $, 
$\l=
\chi^+
=\mu^+$ then we also may omit them, so the typical case
$(st)^1_{\k_,\mu}$ is the existence of a family $\mathcal A\su
[\mu]^\mu$ of cardinality $\mu^+$ which is a \emph{$\k$-family},
that is, the intersection of any two distinct members of $\mathcal A$
has cardinality $<\k$.

Why is using  $(st)^1_{\k,\mu}$ reasonable when $\beth_m(\k)\le
  \mu<\beth_{m+1}(\k)$? The history of this question is rich. 
  In particular, Baumgartner got by forcing,
  without using large cardinals, the consistency of $(st)^1_{\k,\mu}$ with 
  $\k=\k^{<\k}<\mu <2^\k$, so  here $m=0$. 
  
  We are, however, 
   interested in the cases $m\ge 1$, which are closely related
to pcf problems.  

Consider the pcf statement, 

\begin{itemize}
\item[$(*)^2_{\k,\mu,\chi,\l}$] $\k<\mu<\chi<\l=\cf(\l)$ and there is a
  sequence $\ov {\mathfrak a}=\lng \mathfrak a_i:i<\mu\rng$ of finite sets of regular cardinals with each ${\mathfrak a}_i
  \su (\mu,\chi)$
  and such that  $\l=\max\pcf(\bigcup_i\mathfrak a_i)$ and $J_{<\l}[\overline {\mathfrak a}]=\{u\su \mu: \pcf (\bigcup_{i\in u}{\mathfrak a}_i)\su \l\}$  
  (so really $\chi \gg \mu$. The main case, and the one  we shall deal with, for transparency,  is $\l=\chi^+$.)
\end{itemize}

Why  $(st)^1_{\k,\mu,\chi,\lambda}$ and $(*)^2_{\k,\mu,\chi,\l}$ are related to each other and to graph colorings?

\begin{itemize}
\item[$(*)_0$]
 if $\mathcal A\su [\chi]^\mu$ has cardinality $>\chi$,
  and is a $\k$-family, $\k\le \mu\le \chi$ then the natural  bi-partite graph associated to 
  $\mathcal A$ and denoted $G_{\mathcal A}$, 
  (see definition \ref{graphfromfamily} below)
  has coloring number $\ge \chi^+$ . 
\end{itemize}

So finding such $\mathcal A$ with small list-chromatic number, say
$\ku$, with  $\beth_n(\ku)\le \l =\chi^+$, will give consistent lower bounds, 
which is the purpose of this note. The main point here is that the list-chromatic number of such graphs can be lowered by applying the internal forcing axiom from \cite{sh:546}  (see also  \cite{sh:1036}),  a natural generalization of MA.

Observe that 

\begin{itemize}
\item[$(*)_1$]
For $ {\ell} = 1,2 $,
\end{itemize}
\begin{itemize}
\item[(a)]
 if
$(\stt)^\ell_{\k_1,\mu_1,\chi 
}$ and $\k_1\le \k_2\le
  \mu_2\le \mu_1$ then $(*)^\ell_{\k_2,\mu_2,\chi}$. 

\item[(b)]  
if $ (\stt)^{\ell}   
_{\kappa _1, \miq_1 , \xig_1, \lambda_1 } $ 
and 
$ \kappa _1 \le \kappa_2 < \mu _2 \le \mu_1 $ 
and 
$  \xig_1 \le \xig_2 < \lambda _2 \le \lambda _1 $ 
then 
 $ (\stt)^{\ell}   
_{\kappa _2, \miq_2  \xig_2, \lambda_2 } $. 

\item[$(*)_2$] For $ {\ell} = 1,2 $,
\begin{itemize}
\item[(a)] $(\stt)^2_{\k,\miq,\xig}$ 
implies $(\stt)^1_{\k,\miq,\xig}$
and 
 $(\stt)^2_{\k,\miq,\xig, \lambda }$ 
 implies $(\stt)^1_{\k,\miq,\xig , \lambda }$. 
\item[(b)] If $(\stt)^{\ell}_{\k,\miq,\xig}$ 
and $\xig=\xig_1^+$, $\mu_1=\min\{\miq,\xig_1\}
  \ge \k$ (so $\ell=1
\wedge \miq = \xig 
\imply \miq_1=\miq$) then $(\stt)^\ell
  _{\k,\miq_1,\xig_1}$. 
\item[$(c)$] If $(\stt)^\ell_{\mu 
\k, \miq \xig  
,\l}$ 
and $\miq < \xig $  
  and $\xig$ is a limit cardinal of  
cofinality $\not=\cf(\miq)$ 
and $ \not= \cf(\lambda $  
 then for every large enough 
   $\xig_1< \xig $ 
we have 
$(\stt)^\ell_{\k,\miq,\xig_1,\l}$. 
\end{itemize}
 \end{itemize}
 \noindent Also 
 \begin{itemize} 
\item [$(*)_3$] If $\miq^{<\k}<\l=\cf(\l),\, 
 \k=\cf(\k)>\aleph_0$ and 
$(\stt)^1_{\k,\miq,\xig,\l}$ 
then
      $(\stt)^2_{\k,\miq,\xig,\l}$. 
\end{itemize}

Why? By  \cite{sh:410}, 6.1 with $\l,\xig,\miq,\k^*$
 here substituting $\miq^*,\miq,\k,\sigma$ there. 
Similarly, by Theorem 6.2 in \cite{sh:410} 
we have $\oplus^1_m \iff \oplus^2_m$ for all $m\ge 1$. 

Let 
\begin{itemize} 
\item[$(*)^{0,n}_{\k,\mu}$] $\mu\in (\beth_n(\k),\beth_{n+1}(\k))$.
\end{itemize}

So the problem with the consistency of
 $(\stt) 
^1_{\k,\mu , \xig } + 
(\stt) 
^{0,n}_{\k,\miq}$ is having $(\stt) 
^2_{\k,\miq,\xig} +
(\stt)  
^{0,n}_{\k,\miq}$.

We may note that 
clause (b) is justified by the RGCH 
 and $\l=\cf(\l) $ is usually natural. 

An example, then,  of how this note clarifies the question of whether the upper
bound of $\beth_\om(\ku)$ is tight is: 

\begin{conclusion} \label{ABC} We have $(A)\iff (B) \iff (C)$ where:

\begin{enumerate}
\item[(A)] For every $n$ in some forcing 
extension of  $\bf V$  there
  are $\k$, $\miq 
=\beth_n(\k)$, 
$ \xig > \miq $ 	
and a $\k$-family
  $\mathcal A\su
[\xig ]^\miq $ 
 of cardinality 
$ > \xig $.	

\item[(B)] For every $n$ in some forcing extension of $\bf V$ there are
  $\k$,
$ \miq = 
\beth_n(\k)$ and a set $\mathfrak a$ of 
$ \miq $ 
 finite sets of  regular cardinals
$ \miq $ 
such that  
$J_{<\sup (\mathfrak a)}[\mathfrak a]  
\su [\mathfrak a]^{<\k}$, i.e.  
$\mathfrak b\in [\mathfrak a]^\k$ 
    implies that 
$\max\pcf (\mathfrak b)\ge \sup (\mathfrak a)$. 

\item[(C)] For every $n$ in some forcing extension of $\bf V$ there are
  $\k$, 
$ \miq = 
=\beth_n(\k)$ and a graph $G$ with list-chromatic
  number $\k$ and coloring number 
$ >  \miq $. 
\end{enumerate}
\end{conclusion}

\begin{proof} 
[Proof of \ref{ABC}] $(A)\implies (B)$ follows from \cite{sh:410}, 6.1  (and $(B)\implies (A)$ is obvious by $(*)_2$ above). 

$(A)\implies (C)$ is done in Theorem 2.1 below, where we start letting $\theta=2^\k$, or, if $\k$ is regular, also $\theta=2^{<\k}$ suffices to get the assumptions of 2.1.
For every $n$, $(A)_n\iff (B)_n$ and $(A)_{2n+2}\imply (B_n)\imply(C)_n$.

To prove $(C)\implies (B)$ it suffices to  note, (use $\theta=\theta^{<\theta}$) 
that $(a)_{\l,\theta,\k}\imply (b)_{\l.\theta,\k}$ in Claim \ref{comb}. See \cite{LeSh:527} and use the  proof of compactness in   singulars \cite{sh:266} and \cite{sh:668}, Section 2.  
\end{proof}

In conclusion,  the upper bound $\beth_\om(\ku)$ cannot be lowered without making 
 substantial  progress in pcf theory. If, on the other hand, the negations of $\oplus^2_m$ are consistent for all $m$, then Kojman's $\beth_\om(\ku)$ upper bound is optimal.


\subsection{Should we expect consistency or better pcf theorems?}


Let us mention first the known consistency results. Only quite recently Gitik \cite{Git15}
succeeded to prove, from the consistency of large cardinal axioms, 
 the consistency of a countable set of regular
cardinals $\mathfrak a$ with 
$\pcf (\mathfrak a)$ uncountable, but really
 just $|\pcf (\mathfrak a)|=\aleph_1$.
 In particular he got $(\stt)  
^2_{\aleph_0,\aleph_1,\mu}$. While a great 
achievement, this is still very distant from what we need. 

For $\k>\aleph_0$ there are  no 
known consistency results. 
After  the RGCH was proved in the 
early nineties much effort (at
least by the present author) was made 
 to lower $\beth_\om$ --- so far without any success. However  
 in some other
	directions
 there were advances (\cite{sh:824,sh:898,GiSh:1013}).

	So do we expect consistency or
 ZFC results? Wishful thinking, or, if you prefer, 
	the belief that ``set theory behaves in an interesting way"  
	 suggests that  the answer to ``for which $m$ 
	 the statement $\oplus^1_m$ holds in ZFC" 
	 should  turn out to be somewhere in the middle,
	  e.g.   $m=4$ (or $m=957$, for that
	matter). More seriously, the situation is wide open. Perhaps, as on
	the one hand the ZFC $\beth_\om(\ku)$ gap has not changed for a long time
	now, while on the other hand there has been a recent breakthrough in
	consistency, there is some sense in viewing  consistency as more
	likely.

	\section{Proofs}

	\begin{theorem} Suppose that 
	$\k<\theta=\theta^{<\k}$, 
$ \xig > \miq	=\beth_{2\ell+1}(\theta)$ and there 
	a $\k$-family ${\mathcal A}\su
	[\xig]^\miq $ 	
	 of size 
	$ \xig^+ $ 	
{\em Then} in some forcing extension  there is 
a graph $G$ with list-chromatic number $\theta$ and coloring number $>\beth_\ell(\theta)$. 
\end{theorem}

As promised in the introduction, 
we may prove 
 a weak version of the 
{\rm RGCH} theorem from a bound on list-chromatic 
number. 

\begin{theorem} Suppose that $\k<\theta=\theta^{<\k}$,
$ \xig > \miq	
> \beth_{\om}(\theta)$ and there is
a $\k$-family $\mathcal A\su 
[ \xig]^\miq $ 	
 of size 
$ \miq^+ $.	
Then in some forcing extension  there is 
a graph $G$ with list-chromatic number $\theta$
 and coloring number $\ge 
\miq	
>\beth_\om(\theta)$. 
\end{theorem}

\noindent\textbf{Convention}: For this section we fix
  $\aleph_0\le \k<\theta$. 

We shall need the following definition from 
 \cite{sh:546} p. 5. (See also \cite{sh:1036} for
 more on this and other forcing axioms). We  recall one
 implication from 2.4 below: if $\mathcal A$ is as in 2.1
 then $G_{\mathcal A}$, define in 2.4,
 has coloring number 
$ > \miq $. 

\begin{definition} \label{b9}
 A forcing notion $\mathbb P$
 satisfies $*^\omega_\mu$
where $ {\aleph_0} 	
<\mu=\cf(\mu)$ if Player I 
(the "completeness" player) has 
a winning strategy in the following game in 
$\omega$ moves:

At step $k$: If $k\not=0$ then 
Player I chooses $\lng p^k_{1, \alpha 
}:\alpha 
<\mu^+\rng$ with
 $p^k_{1, \alpha 
}\in \mathbb P$ such 
that for all $\xi<\z$ and for 
club-many $ \alpha 
<\mu^+$ in 
 $S ^{\mu^+}_\mu$,  $p^{k-1}_{2, \alpha 
}\le p^k_{1, \alpha 
}$,
 and also chooses a function 
 $f_k:\mu^+\to \mu^+$ which is 
regressive on a club of $\mu^+$. If $k=0$
  Player  I 
chooses $p^0_{1, \alpha i
}=\emptyset_{\mathbb P}$
 and $f^0$ as the identically 
  $0$ function on $\mu^+$. 

Player II chooses $\lng p^k_{2,i}
:i<\mu^+\rng$ such that for club 
many $\alpha 
<\mu^+$ in $S^{\mu^+}_\mu$
 it holds that $p^k_{1,\alpha 
}\le p^k_{2,\alpha 
}$. 

Player I wins 
play of 
if there is a club $E\su \mu^+$
 such that for all 
$ \alpha < \beta $ 
 in 
$E\cap S^{\mu^+}_\mu$,
 if $f_k 
(\alpha 
)=f_k(\beta 
)$ for all  
$k <\omega$ then there is an upper
 bound in $\mathbb P$ to the set 
$\{p^k_{1.\alpha 
}:k<\omega\}\cup\{p^k_{2,\beta 
}:k<\omega\}$. 
\end{definition}

\begin{definition} \label{graphfromfamily} 
\begin{enumerate}
\item $\mathcal A$ is  a $\k$-family 
 of sets when $|A\cap
B|<\k$ for all distinct 
$A,B\in \mathcal A$ and is a 
$(\theta,\ka)$-family if in 
addition, $|A|\ge\theta$ for all  $A\in \mathcal A$ and, for notational transparency, $\mathcal A\cap \bigcup \mathcal A=\emptyset$.

\item Suppose $\mathcal A$ is a $\k$-family of sets.  
We define the 
(bipartite) graph
 $G_{\mathcal A}$. Its 
  set of 
vertices  is
$V_{\mathcal A}=\mathcal A\cup  
\bigcup \mathcal A$. We denote
$\bigcup \mathcal A$ by 
$pt (\mathcal A)$. The edge set $E_{\mathcal
  A}$ is $\{\{v,A\}: v\in A\in \mathcal A\}$.
 When $\mathcal A$ is fixed or clear from context, 
   we refer to  $G_{\mathcal A}$ as  $\lng V,E\rng$. 
\end{enumerate}
\end{definition}

\begin{definition} \label{b15}
For a $(\theta,\k)$-family $\mathcal A$, a set $Y\su G_{\mathcal A}$ is \emph{closed} (pedantically, $\mathcal A$-closed, but the identity of $\mathcal A$ is clear from the context) if:
\begin{enumerate}
\item  [(a)]  
$A\not= B\in Y\cap \mathcal A\imply A\cap B\su Y$.
\item [(b)] 
If $A\in \mathcal A$ and $|A\cap Y|\ge \k$ then $A\in Y$. 
\end{enumerate}
A sub-graph $G'$ of $G_{\mathcal A}$ 
is closed if its set of vertices is closed. 
\end{definition}

\begin{claim} 
If $\mathcal A$ is a $(\theta,\k)$-family 
and $\l^\k=\l\ge \theta$ then every sub-graph of
 $G_{\mathcal A}$  of size $\l$ is contained in a
closed sub-graph of the same size.
 Moreover, if $Y_1\su G$ is closed
and $X\su Y_1$ is of size $\l$, 
then 
there is a closed $Y\su Y_1$ of
cardinality $\l$ such that $X\su Y$. 
\end{claim} 

Remark: instead of $\l^\k=\l$ it suffices 
 that $\mathcal D(\l,\k)=\l$, where $\mathcal
D(\l,\k)=\cf([\l]^\k,\supseteq)$, (see \cite{koj}).

\begin{definition}
  Suppose $\theta> \k$ and $\mu$ are  cardinals and 
  ${|\a|^\kappa}<\mu$ for all $\a<\mu$. We
  say that \emph{$Pr_{\theta,\k}(\mu)$ holds}
{\em when} 
for every $(\theta,\k)$-family $\mathcal A$ and
  every [closed] $Y\su G_{\mathcal A} $ of 
cardinality $|Y|<\mu$ the 
  list chromatic number of
 $Y$ is at most $\theta$; that is, for every 
  assignments of lists $L(v)$ to vertices in 
$G_{\mathcal A}$ such
  that $|L(v)|\ge \theta$ there is a 
valid coloring $c\in \prod_v
  L(v)$.
\end{definition}

\begin{claim} \label{3A} Assume that 
$ {\mathcal A} $ is a 
 $ ( \theta , \kappa )$-family, 
$Y\su G_{\mathcal A}$ is closed, 
$\cf(\delta)\not=\cf(\k)$, $\d<\cf(\mu)$ 
and $Z_i\in [Y]^{<\mu}$ 
is 
increasing with  $i<\d$.  
If each $Z_i$ is $\mathcal A$-closed  then
 $Z:=\bigcup_{i<\d} Z_i$ is $\mathcal A$-closed. 
\end{claim}

\begin{proof}
First, if $A \not= 
B\in Z\cap \mathcal A$ then for 
some $i<\d$ it holds that $A,B\in Z_i$,
but $ Z_i $ is $ {\mathcal A} $-closed 
hence $ A \cap B \subseteq Z_i \subseteq Z $ 
as required (in \ref{b15}(a)). 
 Second, 
if $A\in \mathcal A$ satisfies that 
$|A\cap Z|\ge \k$ then
(because $ \cf( \delta ) \not= \cf ( \kappa ) $)
 for some 
$i<\d$ it holds that $|A\cap Z_i|\ge \k$
 and as $Z_i$ is closed, $A\in Z_i\su Z$.
\end{proof}

\begin{lemma}[Step-up Lemma] \label{stepup} 
Suppose that $ \lambda > 
\mu=\mu^{<\mu}>\k$ and $\theta>\k$.
Assume that 
\begin{enumerate} 
\item [(a)] 
The internal forcing axiom for posets that satisfy
  $*^\omega_\mu$ 
(see Def \ref{b9})	
holds for $<\l$ dense sets.
\item [(b)] 
 $(\forall \a<\mu)(|\a|^\k<\mu)$. 
\item [(c)] 
 $(\forall \a<\l)(|\a|^\k<\l)$
\item [(d)] 
 $Pr_{\theta,\kappa}(\mu)$ holds.
\end{enumerate}
Then $Pr_{\theta,\kappa}(\l)$. 
\end{lemma} 

\begin{proof} Suppose 

$(*)_1$ $\mathcal A$ is a $(\theta,\kappa)$-family and 
  $Y_*\su G_{\mathcal A}$ is closed,  $|Y_*|<\l$ and 
  $L(v)$ such that $|L(v)|= \theta$ 
 is given for all $v\in Y_*$. 

We need to
  prove the existence of a valid coloring $c$ of $G$ such that
  $c(v)\in L(v)$ for all $v\in Y_*$.

  Let $\mathbb P$ be the following poset:
$q\in \mathbb P$ iff $q$ is a partial valid
  coloring 
for	
the given lists and 
$\dom (q)\su
 Y_*  \su  G_{{\mathcal A}} $ 
 is closed of
  cardinality $<\mu$. A condition $q$ 
is stronger than a condition
  $p$, $q\ge p$, iff $p\su q$.

$(*)_2$ $\mathbb P$ is a forcing notion. 

$(*)_3$ (Density) if $p\in {\mathbb   P} $ 
and $Z\su Y_*$ satisfies $|Z|<\mu$
{\em  then} 
there is $q\ge p$ such that $Z\su \dom (q)$. 

\medskip

{Proof of $(*)_3$:} By  possibly 
increasing $Z$, we may 
assume that $Z$ is closed in $Y_*$ and
that $\dom (p)\su Z$. As $\dom (p)$ is 
closed, for all $A\in \mathcal A
\cap Z\sm \dom (p)$ it holds that
 $|A\cap \dom (p)|<\k$. For 
$A \in {\mathcal A} \cap 
 Z\sm
\dom (p) $ 	
 let
$L'(A)=L(A)\sm \{ p(v):v\in A\cap \dom (p)\}$.
 As $|A\cap \dom
p|<\k<\theta$ it holds that $|L'(A)|=\theta$. 

For $v\in (Z\sm \dom (p))\cap pt(\mathcal A)$ 
there is at most one $A\in \dom (p)$
such that $v\in A$, because $\dom (p)$ is closed. Let $L'(v)$ be 
gotten from $L(v)$ by subtracting
$\{p(A)\}$ 
from $L(v)$ for that 
unique $A$, when $A$ exists. For all
$x\in \dom (p)$ let $L'(x)=L(x)$.

By $Pr_{\theta,\kappa}(\mu)$ there
 is condition $p''$ with $\dom (p'')=Z$ such that
$p''(v)\in L'(v)$
for every $ v \in Z $. 
 Let $p'=p''\rest 
(Z\sm \dom ( p)) 
$ and let $q=p\cup
p'$.  Now we claim that $q\in 
{ {\mathbb   P} } $. 
As $ x \in Z \imply 
q(x)\in L'(x)\su L(x)$ for all
$x\in \dom (q)$, all that needs to
 be checked is the validity of the
coloring $q$. Suppose that $v\in A$
 and $v,A\in \dom (q)$. First assume
that $v\in \dom (p)$ and $A\in \dom (p')$.
 In this case $p(v)\not=p'(A)$
because $p(v)\notin L'(A)$ by the  
definition of $L'(A)$. Another
case to check is $v\in \dom (p')$ 
and $A\in \dom (p)$, 
 which followed from
the choice of $L'(v)$. The two
 remaining cases are clear,
holds as $ p, p \in {\mathbb   P} $. 

\smallskip
\item{$(*)_4$} If $\lng p_i:i<\d\rng$ 
is an increasing sequence of conditions
in $\mathbb P$ and 
$ \delta < \mu $  and 
$\cf(\d)\not=\cf(\k)$ 
then the union is a condition,
and is an upper bound of the sequence. 

Proof of $ (*)_4 $:
Let $Y_\d=\bigcup\{\dom p_i): i<\d\}$. 
Now $|Y_\d|<\mu$  as $\delta  
<\mu$ by the
assumptions, and $i<\d \imply  
\dom(p_i)\in[Y_*]^{<\mu}\su [G_{\mathcal
  A}]^{<\mu}$ by the  
definition of $ {\mathbb   P} $, 
  recalling that 
$\mu$ is regular (
by $ \mu= 
\mu ^{ <\mu } $  from 
 of the
claim's assumptions). Since  
$\cf (\d)\not=\k$, by \ref{3A} 
 it holds that  $p=\bigcup_i p_i$  is
 a condition. 

\smallskip
\item{$(*)_5$} If $\d<\mu$, 
$\ov p=\lng p_i:i<\d\rng$ is increasing in 
$\mathbb P$ and $\cf(\d)=\cf(\k)$ then  
$\ov p$ has an upper bound in $\mathbb P$. 

Proof of $ (*)_5 $:
 Let $Z\su
Y_* \su  
 G_{\mathcal A}$ be closed such that
$|Z|<\mu$ and $Y=\bigcup_{i<\d} \dom(p_i)\su Z$. 
  By restricting to a subsequence we may
 assume that $\d=
\cf(\k)$ 
 and so $A\in \mathcal A\sm Y \imply
\bigwedge_{i<\k} (|A\cap \dom (p_i)|<\k
 \imply |A\cap Y|\le \k$. Now repeat
the proof of $(*)_3$  with $p=\bigcup_i p_i$ with the
following  changes: 

\begin{enumerate}
\item[$(a)$]  if $A\in Z\sm Y$, $A\in \mathcal A$, then $|A\cap Y|\le
\k$ hence $L'(A)=L(A)\sm \{p(v):v\in A\cap Y\}$ has cardinality
$\theta$ as $L(A)$ has cardinality $\theta > \k\ge  |A\cap Y|\ge |\{p(v):
v\in A\cap Y\}|$.

\item[(b)] if $v\in Z\sm Y$, $v\in
 pt (G_{\mathcal A})$, then 
  \[i<\k\imply |\{ A\in \dom (p_i)\cap ]
\mathcal A: v\in A\}|\le
  1,\] 
hence 
\[|\{ A\in \mathcal A\cap Y: v\in A\}|\le 1\]
and  $L'(v)=L(v)\sm \{p(A): A\in Y_\d 
\wedge v\in A\}$ has cardinality
$\theta$. 
\end{enumerate}

Now we can 
finish	
as in $(*)_3$.

\smallskip
\item{$(*)_6$} $\{ p^\ell_\zeta: 
 \ell=1,2 \;\text{and }  \z<\d\}$ has a common upper
bound {\em when}: 
\begin{enumerate} 
\item[$(a)$] $\d<\k^+\le \mu$ 
(we will use $\d=\omega<\k^+$ when 
simpler 
as this us the one we shall use).	2016-02-09 simpler). 
\begin{enumerate} 
\item[$(b)$] $p^\ell_\z\in \PP$ 
\item[$(c)$] $\z<\xi<\d\imply p^\ell_\z\le_{\PP} 
p^\ell_\xi$.
\item[$(d)$] $p^1_\z,p^2_\z$ are compatible 
functions for $\z<\d$. 
\end{enumerate} 
\end{enumerate} 

Proof of $ (*)_6 $: 
 Let $p=\bigcup_{\ell,\z}p^\ell_\z$, 
 so $p$  is a function, but not
 necessarily a 
condition 
in $\PP$. 
 Let $Y=\dom (p)$ and $Z\supseteq Y$ be
closed  and of cardinality $<\mu$
such that $ Z 
\subseteq  Y_* 
\subseteq G_ {\mathcal A} $.

(A) If $A\in Z\sm Y$, $A\in \mathcal A$,
and $ {\ell} \in \{ 1,2 \} $          
 then 
$\ell \in \{1,2\}\wedge \z<\d \imply 
|A\cap \dom (p_ \zeta 
^\ell)|<\k$ so 
$\lng |A\cap \dom (p_\zeta 
^\ell)|:
\zeta 
<\d\rng$ is a non-decreasing sequence of 
sets each of cardinality 
$<\k$  hence $\le \k$. 
So 
 $|A\cup
 \bigcup_{\zeta  
} \dom (p_\zeta  
^\ell)|\le \k$, 
hence $|A\cap Y|\le \k$.

(B) If $v\in Z\sm Y$, $v\in pt (G_{\mathcal A})$, 
and $ {\ell} \in \{ 1,2 \} $ 
then $|\{A\in \bigcup_{\zeta 
} \dom(p_\zeta  
^\ell), v\in A\}|\le 1$ hence 
$|\{A\in Z\sm A: A\in \mathcal A,
 v\in A\}|\le 2$.  So all is fine.

We
continue as in the proof of $(*)_5$. 

\smallskip
\item{$(*)_7$} $\mathbb P$ is $\mu$-complete  (by $(*)_4 + (*)_5$). 

\smallskip
\item{$(*)_8$} The property $*^\omega_\mu$ holds for $\mathbb P$. 

The game which
defines $*^\omega_\mu$ lasts
 $\omega$ steps and at each step
$k<\omega$ 
for $ {\ell} = 1,2 $ 	
a sequence of conditions 
$ \lng	
p^ k _{1, \alpha  } : \alpha 
<\mu^+\rng$, a club
$ E_k \su \mu^+$ and a regressive 
function $f_ k $  
with domain 
$\mu^+ $ 
played by the completeness
player I (see Def \ref{b9} or  
 \cite{sh:546} p. 5. See also \cite{sh:1036}
 for more on this 
and related 
forcing axioms).

This is how player I chooses $E_k$ and $f_k$: 
 $E_k$ is
sufficiently closed; 
$f_k$ has domain $ \mu ^+ $ and 
 is regressive such that:

$\oplus$ If $\a_1,\a_2\in 
\dom (f_k)\cap 
S^{\mu^+}_\mu \cap E_k$, 
$f_k(\a_1)=f_\z(\a_2)$ then 
$p_{1,  \a_1}^k$, $p_{1, a_2}^k$  
 are compatible functions. 

This clearly suffices (as the   
$\lng (p_{1,  \a_1}^k,p_{1, \a_2}^k):   
 k<\d\rng$ 
are like $\lng (p^1_\z,p^2_\z)$ in $(*)_6$).

Clearly such a function
$ f_k $   
 exists
but we elaborate. 

$(*)_{8.1}$ $f_\zeta(\delta)$ codes:
\begin{enumerate}
\item [$(a)$]
$a^k_\d=\dom(p^k_{1, \delta})  
\cap (\bigcup_{\a<\delta} \dom
p^k_{1, \alpha } ) $.   
\item [$(b)$]
$p^k_{1, \delta } 
\rest a^k  
_\delta$.

\item[(c)]
$ \{ 
p^k_{1, \delta} (v ): v \in 
\dom (p^k_{1, \delta } ) \cap \cap  (
{\mathcal A} )
\text{ and }
p^k_{1, \delta } (v) \in 
\bigcup _{ \alpha < \delta }  \ran ( p^k_{1,
\alpha } ) 
\} $ 

\item[(d)]
$ \{ 
p^k_{1, \delta} (A ): A \in 
\dom (p^k_{1, \delta } ) \cap {\mathcal A} 
\text{ and }
p^k_{1, \delta } (A) \in 
\bigcup _{ \alpha < \delta } \ran ( p^k_{1,
\alpha } ) 
\} $ 

\end{enumerate}

[ What is the point of clauses (c), (d)?
Consider $ p^k_{1, \alpha _1 } 
, p^k_{ 1, \alpha _2 } $
with $ \alpha _1 < \alpha_2 $;
maybe there are 
$ {\ell} \in \{ 1,2 \} , v \in \dom
 ( p^k_{ 1, \alpha _
{\ell} } ) 
\cap   pt 
 ( G_ {\mathcal A} ) ,
A \in \dom ( p^k_{ 1 , \alpha _{3- {\ell}}} ) 
\cap 
{\mathcal A} $
such that
$ v \in A $ and
$ p^k_{ 1, \alpha _ {\ell} } 
( v )  =  
 p^k_{1, \alpha _{3- {\ell}}}( A )
$.
Those are avoided by those clauses. 2016-02-10 s

So now player I wins as whenever 
 $ \alpha < \beta $ 		
belong to 
$S^{\mu^+}_ \mu 
\cap
\bigcap _{k} E_k$ and $\bigwedge_{k} 
f_k(\alpha )   
=f_k(\beta ) $,  		
 the set of
conditions $\{p^k_{1,\alpha   
}:k<\om\}\cup
\{p^k_{1, \beta  
}:k<\om\}$ has an upper bound in 
$ {\mathbb   P} $
by $ (*)_6 $.  	

This proves $(*)_ 8 $. 

$(*)_{9}$ 
If $ x \in Y_* $ then $ D_x $
is an open dense subset of $ {\mathbb   P} $ 
where

$ D_x= \{ p \in {\mathbb   P} : x \in 
\dom ( p ) \} $ 

[ Why $ ( * ) _ 9 $ holds? by $ ( * ) _3 $.]

By the axiom for posets with  
$*^ \omega _ \mu $, 
there is a 
generic filter  for $\mathbb P$ which meets 
all dense sets $D_x$ for $x\in Y$, 
where $p\in D_x$ if $x\in\dom(p)$. 
The  union of the generic  is  a 
valid coloring from the lists on
$ Y_* $.	
\end{proof}

\begin{corollary} \label{twoten}
Suppose $n\ge 1$ and
\begin{enumerate}
\item [$(a)$]
$\mu_0<\mu_1\dots <\mu_n$.
\item [$(b)$]
 For all $l\le n-1$ it holds that $(\forall
  \a<\mu_\ell)(|\a|^\k<\mu_\ell)$. 
\item [$(c)$]
$ \mu_ {\ell} 
^{< \mu _ {\ell} }=\mu_{\ell}$ and 
$2^{\mu_i}=\mu_{i+1}$ for $i<n$. 
\item [$(d)$] 
For every ${\ell} 
<n$, the forcing axiom for posets with
  $*^\omega
_{\mu_{\ell} } $ 
and $<\mu_{{\ell} 
+1}$ dense sets holds. 
\item [$(e)$]
$\mu_0\le\theta^+$
and $ \kappa < \theta $.		 
\end{enumerate}
Then $Pr_{\theta,\kappa}(\mu_n)$. 
\end{corollary}

\begin{proof} By induction on $n$. 
Since the list-chromatic number of
  any graph $G$ of cardinality $<\mu_0$ 
is $\le |G|\le\theta$, the
  condition $Pr_{ \theta, \kappa }     
(\mu_0)$ 
holds trivially.
 The induction step follows
  from the main lemma  
\ref{stepup}.  	
\end{proof}

Next we show how to force the  
conditions of the previous lemma.

\begin{claim} \label{AtoB} Assume that: 
\begin{itemize}
\item[(A)]$\theta=\theta^{<\k}>\k$ and $\beth_{2n+1}(\theta)<\mu\le \chi<\l$. 
\item[(B)] $(st)^1_{\k,\mu,\chi,\l}$.
\end{itemize}

Then: 
For some $\mathbb P$ and $ \bar{ \mu } $ 
\begin{itemize}
\item[(a)] $\mathbb P$ is a 
$\theta^+$-complete forcing notion 
that satisfies $(\beth_{2n+1}(\theta))^+$-c.c. 
\item[(b)] 
$ \bar{ \mu } = \langle \mu _{\ell} 
:{\ell} \le n \rangle $ 
in ${\bf V}^{\mathbb P}$
such that 		
 $\mu_\ell= \mu ^{\mu _ {\ell} } = 
(\beth_\ell(\theta))^+
<\beth_{\ell+1}(\theta)$ for all 
$\ell\le n$, $|\a|^\k<\,\mu _ {\ell} $ 
for all $\a<\mu _  {\ell} $ 
and 
$\beth_n(\theta)<\mu\le \chi<\beth_{n+1}(\theta)$.
\item[(c)]  $(\stt)^1_{\k,\mu,\chi,\l}$.  
\item[(d)]  The forcing axiom 
$*^\om_{\mu_\ell}$ with $<\mu_{\ell+1}$
 holds for all $\ell\le n$. 
\end{itemize}

\end{claim}

\begin{proof}
Now, 

\begin{itemize}
\item[$(*)_1$] $(\beth_{2n+1}(\theta))^+
\le \miq $ and 	
 $(\beth_{2n+1}(\theta))^+ <  \xig $ 	
\end{itemize}

Why $ (*)_1$? 
The first inequality hold by (A). 
For the second,  
letting
$ \chi_1 = \beth _{2n+1}( \theta )^+ $
we have
$ \chi_1 ^{ \beth _{2n}( \theta )}= \chi_1$ 
hence $ \chi_1 ^ \k = \chi_1 $
whereas
$ \xig		
^ \k \ge  \lambda > \mu $ because
we are assuming 
$ ({\stt})^1_{\k, \mu, \chi,\lambda } $ 

Now let 

\begin{itemize}
\item[$(*)_2$] \begin{itemize}
\item[(a)] $\mu_\ell=(\beth_{2\ell}(\theta))^+$ for$\dagger$ 
$\ell\le n$ so $2^{<\mu_\ell}\le \beth_{2\ell+1}(\theta)$.
\item[(b)] Choose 
$ \mu_{n+1} $ such that   
$\mu_{n+1}=\cf(\mu_{n+1})=(\mu_{n+1})^{(\mu_n^\k)}>\l$ s
and 	
$\a<\mu_{n+1} \imply |\a|^\k <\mu_{n+1}$	
\end{itemize}
\end{itemize}

Remark: less suffices; 
 $\mu_{n+1}=(\l^\k)^+$ or just $\mu_{n+1}=\cf(\mu_{n+1})>\l$
 satisfies $(\forall \a<\mu_{n+1})
(|\a|^\k<\mu_{n+1})$, but will complicate the
  notation below, e.g. $(*)_4(b)$ for 
$\ell=n$ is different. 
  
  Now 
  \begin{itemize}
  \item[$(*)_3$] \begin{itemize}
  \item[(a)] $\mu_0=\theta^+$ hence 
$\mu_0=\cf(\mu_0)$ and 
  $(\forall \a)(\a<\mu_0\to |\a|^\k\le 
\theta^\k  = \theta 
<\mu_0)$. 
  \item[(b)] $\mu_0<\mu_1<\dots <\mu_n<\mu_{n+1}$ 
are regular. 
  \item[(c)] $(\forall \a<\mu_\ell)
(|\a|^\k<\mu_\ell)$ for all $\ell\le n+1$.
    \item[(d)] $(\mu_{\ell+1})
^{2^{<\mu_\ell}}=\mu_{\ell+1}$.
  \item[(e)]  $\mu_n<\mu\le \chi<\l<\mu_{n+1}$.
  \end{itemize}
  \end{itemize}

Let 

\begin{itemize}
\item[$(*)_4$] Let 
\begin{itemize}  
\item[(a)] $\mathbb Q^*_\ell=
{\rm Levy} 
(\mu_\ell, 2^{<\mu_\ell})$ for $\ell\le n$. 

\item[(b)]
 $\mathbb Q_*=\prod_{\ell \le n} 
\mathbb Q^*_{\ell} $. 
\item[(c)] $ Q^*_{ \le k }=\prod_{\ell \le k} 
\mathbb Q^*_{\ell} $ 
\end{itemize} 
\end{itemize} 

Easily, 

\begin{itemize}
\item[$(*)_5$] 
\begin{itemize}  
\item[(a)]  
$\mathbb Q^*_\ell$ is $\mu_\ell$-complete 
and of cardinality $2^{<\mu_\ell}$.  
so satisfies the  $2^{<\mu_\ell}$-cc. 
\item[(b)] 
 Let   
${\bf V}_\ell:= {\bf V}^
{\prod_{k<\ell}\mathbb Q^*_\k}$ 
\end{itemize} 
\end{itemize} 


\begin{itemize}
\item[$(*)_6$]  
in 
${\bf V}_{n+1}:={\bf V}^
{\prod_{\ell\le n} \mathbb Q_\ell}$,
we 
 define  $\lng (\mathbb P_k,
{\underset \sim{\mathbb Q}}^2_\ell): 
 k\le n+1, \ell\le n\rng$ 
such that : 
\begin{itemize}
\item[(a)] 
$\mathbb P_0$ is the trivial forcing. 
\item[(b)] 
$\mathbb P_{\ell+1}$ is a forcing 
notion of cardinality 
$\mu_{{\ell} +1} $. 
\item[(c)] 
$\mathbb P_{\ell +1}$ satisfies the $\mu_\ell^+$-c.c. 
\item[(d)] 
$\mathbb P_{\ell+1}=\mathbb 
P_\ell *  
\underset \sim{\mathbb 
{Q}}_\ell^2$. 
\item[(e)] 
${\underset \sim {\mathbb Q}}^2_{\ell}$   
is a ${\mathbb P}_{<\ell}$-name of a forcing
 notion of cardinality $\mu_{\ell +1}$ 
 that satisfies $\mu_\ell^+$-c.c. that forces 
 $2^{\mu_\ell}=\mu_{\ell+1}$ and the 
axiom for forcing notions that satisfy
  $*^\om_{\mu_\ell}$ for 
$<\min\{\mu_{\ell+1}, (
\miq  
^{\kappa})^+\}$ 
dense sets. 
 \item[(f)]   $ {\mathbb   P} _{k+1}$ 
is a $ {\mathbb Q}_* $-name 
and actually a 
$ {\mathbb Q}  ^*_{ \le k } $-name 
\end{itemize} 
\end{itemize} 

There is no problem to carry the  
induction (note that 
$(\mu_{\ell+1})^{<\mu}=\mu_{\ell+1}$ 
in $\mathbf V^{\mathbb P_n+1}_{n+1}$.) 
We return to $\bf V$,  
so we 	
 have a ${\mathbb Q}_{ \le  k}$-names 
so $ \underset \sim{{\mathbb   P}} _{\ell} $ for  
$ {\ell} = 0,\dots, k+1 $ 
for the forcing notion 
above.  
Let, in $\bf V$, 
$\mathbb P^{k+1} 
= \mathbb {\mathbb Q}  ^*_{\le k } *  
{\underset \sim{\mathbb   P}}_{k+1}$. 
 Why ${\mathbb P} = {\mathbb  P} ^n 
$ is as required?

Clearly,   
all forcing notions  
${\mathbb Q}^*_\ell, \mathbb Q_*, 
{\mathbb Q}  ^*_{ \le k}, {\mathbb   P} ^k$ , 
are $\theta^+$-complete, hence
in particular 
 so is 
$ {\mathbb P}$.
 Therefore,
in $ {\bf V }  ^ {\mathbb   P} $ still 
 $(\forall \a<\mu_\ell)
(|\a|^\k<\mu_{\ell+1})$ for all $\ell<n+1$ 
because we prove below 
 that $\mu_\ell$ does not collapse.
 
 Clearly, ${\mathbb P} ^{k+1} 
$ has cardinality $\mu_{k+1}$ and 
satisfies the $ ( 2^{< \mu _{k}})^+ $-cc 
and in $ {\bf V } $ the forcing notion 
$ { {\mathbb Q}}  _{\ell} $ 
are $ \mu _ {\ell} $-complete 
and in $ {\bf V }  ^{ {\mathbb Q}  ^*_{\le k } } $ 
the forcing notion $ {\mathbb Q}  ^2_{\ell} $ 
is forced to be 
$ \mu _ {\ell} $-complete.  
Hence in $ {\bf V } $ 
for $ {\ell} \le k $ we have 
$\force_{{\mathbb P}^k}"  
\mu_\ell =\mu_\ell^{<\mu_\ell}$
 is not collapsed", and ${\mathbb P}^k$ 
satisfies the $((2^{<\mu_k})^+)$-c.c. 
as ${\mathbb  Q}^*_{ \le k } $ 
does, and ${\mathbb 
P}_{k+1}$ satisfies 
$(\mu ^ \kappa )+ $-c.c.. 
 
 Lastly, the relevant forcing axiom holds: if $\ell < n$, the one 
 for $(*)^\epsilon_{\mu_\ell}$ and $<\mu_{\ell+1}$-dense sets. So replacing
 $\mu_{n+1}$ by $(\mu^\k)^+$ and applying 2.6  we are done. 	
 \end{proof}
 
A similar argument works to replace $n$ with $\om$: 
 
\begin{theorem}\label{bethomega}
The condition $(A)_{\ell(*)}$ implies the condition $(B)_{\ell(*)}$ for $\ell(*)\in \{1,2\}$, where: 
\begin{enumerate}
\item[$(A)_1$] $\aleph_0<\cf(\k)\le \kappa<
\theta=\theta^{<\k}$,  $\chi\ge \l\ge \beth_\om(\k)$ and 
there exists a  $\k$-family $\mathcal A\su 
[\chi]^\l$ of cardinality $|\mathcal A|\ge \chi^+$. 
\item[$(A)_2$] $\aleph_0<\cf(\k)\le\k<\theta=\theta^{<\k}$ and 
for every  $n<\om$ there are 
  $\chi_n > \l_n \ge \beth_n(\theta)$ 
a $\k$-family $\mathcal A_n\su [\chi_n]^{\l_n}$
 of cardinality $|\mathcal A_n|\ge\chi_n^+$
and $ \lambda _n \notin [ \beth _\omega ( \theta ) 
, \beth _{\theta +1 } ( \theta ) ] $ 
\item[$(B)_1$] For some forcing notion $\mathbb P$ 
not adding new sequences of ordinals
 of length $<\theta$, it holds that:
\begin{itemize} 
\item $(\beth_\om(\theta))^{\mathbf V^{\mathbb P}}
=(\beth_\om(\theta))^{\mathbf V}$. 
\item There exists a graph $G$ with list-chromatic 
number $\theta$ and coloring number $>(\beth_\om(\theta))^+$. 
\end{itemize}
\item[$(B)_2$]  Like $(B)_1$ with the coloring number $\ge (\beth_\om(\theta))^+$. 
\end{enumerate}
\end{theorem}

 \begin{proof}
 
 Stage A. For $(A)_1\imply (B)_1$  assume $(A)_1$ and 
 let 

\begin{itemize}
\item[$ (*)_1$]
$(\chi_n,\l_n)=(\chi,\l)$, $\mathcal A_n=\mathcal A$,
 so we can assume  $(A)_2$. 
\end{itemize}

 \begin{itemize}
 \item[$(*)_2$] Let
 $u_1=\{n: \l_n<\beth_\om(\theta)\}$
hence $ n \in u_1
\imply 
\lambda _n 
 < \beth _ \omega )\theta )$  
 and let 
 $u_3=\{n:\l_n>\beth_\om(\theta)\}$. 
 \end{itemize} 

Recalling clause $ (A)_2 $  
note that $ u_1, u_3 $ is a partition 
of $ \omega $. 
 
 \begin{itemize}
 \item[$(*)_3$] Without loss of generality,  for some $\mathbf i \in \{1,2,3\}$ we have:
 \begin{itemize}
 \item[(a)] $u_{\mathbf i}=\om$. 
 \item[(b)] If $\mathbf i=3$  without loss of generality there is some $\l_*>\beth_\om(\theta)$
 such that $\bigwedge_{n} \l_n=\l_*$. 
 \item[(c)] If $\mathbf i=2$  let $\mu_*=\beth_\om(\theta)$. 
 \end{itemize}
 \end{itemize} 
 
 Stage B. Now
 \begin{itemize}
 \item[$(*)_4$] Without loss of generality  there is a sequence $\lng \mu_n:n<\om\rng$ such that 
 \begin{itemize} 
 \item[(a)] $\mu_0=\theta^+$. 
\item[(b)] $\mu_n=\cf(\mu_n)$.
\item[(c)] $2^{\mu_n}=\mu_{n+1}$ .
\item [(d)] Hence $\sum_n\mu_n=\beth_\om(\theta)$.
 \item[(e)] The forcing axiom $*^\om_{\mu_n}$ and $<\mu_{n+1}$ dense sets holds. 
 \end{itemize}
 \end{itemize} 
 
 Why? As in the proof of \ref{AtoB},
but note 
 that the forcing may (in fact do) 
collapse $ \beth _{\omega + 1 } $ 
to $ \beth _\omega ( \omega ) ^+ $.  
Also in the case $ {\bf i} =1 $ letting 
wlog $ \lambda_n \in [\beth _{k(n)},  
\beth _{k(n)+1 } ( \omega ) ) $     
wlog $ k(n) $ is increasing  	
and $ k(n + 1  ) > k(n)+n  $. 

 \begin{itemize}
 \item[$(*)_5$] Without loss of generality, 
in addition, letting $\theta_\om=(\beth_\om(\theta))$,
 we have $2^{\theta_\om}=\theta_\om^+$ 
and $\mu_{\om+1}=2^{\mu_\om}$ is $>\sum_n \chi_n$ 
 and as in $(*)_4(e)$  the forcing axiom
 $*^\om_{\theta_\om^+}$ and $<\mu_{\om+1}$ dense
  sets holds. 
 \end{itemize} 

Stage C.   We deal with the case $\mathbf i=1$. 

By \ref{twoten}, for every $n$,  
$Pr_{\theta,\k}(\mu_n)$ holds. By easy compactness for singulars 
argument we have, as $\aleph_0<\cf(\theta_*)$,
 also $Pr_{\theta,\k}(\mu_\om)$.

Now clearly for each $n$, 
$ \miq_n < \xig_n $ 	
 as in the proof of Theorem 1, there is a graph $G_n$ 
with $|\mathcal A_n|$ vertices, coloring number $\ge \l_n$ and list-chromatic number $\theta$. 

Taking then the disjoint sum of all $G_n$ we have established $(A)_2\imply (B)_2$. 

Stage D. $\mathbf i\in \{3\}$. 
 Similarly, but we use $(*)_5$.  

 \end{proof}  

\noindent
\textbf{Remark}: We can replace $\beth_\om(\theta)$ 
with $\beth_{\d(*)}(\theta)$ when $\d(*)<\cf(\k)$.

 \begin{proof}[Proof of Theorems 1 and 2] 
The proofs consists of combining the lemmas above. 
\end{proof}

 \medskip
 We conclude with a few simple implications that are needed above. 
 
 \begin{claim} \label{comb} Assume that $\theta$ 
is a regular cardinal and $2^\k\le \theta\le \l$. 
 We have $(a)_{\l,\theta,\k}\imply (b)_{\l,\theta,\k}
 \imply (c)_{\l,\theta,\k}\imply 
 (d)_{\l,\theta,\k}$. If, in addition,
 $\theta=\theta^\k$ (or just 
$ \partial  < \theta \imply \partial ^\kappa < \lambda $ 
 then $(d)_{\l,\theta,\k}\imply 
 (e)_{\l,\theta,\k}\imply (f)_{\l,\theta,\k}$,

  Where
  
\begin{itemize} 
\item[$(a)_{\l,\theta,\k}$]  $\l$ is minimal 
such that there is a graph 
$G$ with $\l$ 	vertices,
coloring number 
$ \ge  \theta $ and 
list-chromatic number $\le \k$. 

\item[$(b)_{\l,\theta,\k}$] $\l$ is regular 
and there is a graph $G$ with $\l$ vertices,
coloring number $\ge \theta$, every sub-graph 
of $G$ with $<\l$ vertices has coloring
number $\le \theta$ and the complete bipartite 
graph $K(\k,2^\k)$ is not weakly 
 embeddable into $G$. 

\item[$(c)_{\l,\theta,\k}$] $\l>\theta$ is 
regular and there is $\ov C$ such that:
\begin{enumerate}
\item[$({\mathbf \a})$] $\ov C =\lng C_\d:\d\in S\rng$ 
\item[$({\mathbf \b})$] $S\su \{\d: \d<\l \wedge 
\cf(\d)=\theta\}$ is stationary.
\item[$({\mathbf \gamma})$] $C_\d\su \d$ and $\otp(C_\d)=\theta$.
\item[$({\mathbf \d})$] If $u\in [\l]^\k$ then 
$\{\d\in S: u \su C_\d\}$ is bounded in $\l$. 
\end{enumerate}
\item[$(d)_{\l,\theta,\k}$] $\l>\theta$ is 
regular and for some $\mu<\l$ for 
every $\partial  \in [\k,\theta)$  
there is $\mathcal A\su [\mu]^\partial $ 
of cardinality $\l$
such that $u\in [\mu]^\k\imply (\exists^{<\l}
 v\in \mathcal A)(u \subseteq v)$ 	
\item[$(e)_{\l,\theta,\k}$] $\l>\theta$ is 
 regular and there are $\mu<\l$ and 
$\{A_\D:\D\in [\k,\theta)\}$ 
such that $A_\D\su [\mu]^\D$ 
 is a $\k$-family of cardinality $\l$. 	
\item[$(f)_{\l,\theta,\k}$] $\l>\theta$ is
 regular and there are $\mu<\l$ and 
$\{\mathfrak a_\D: \D\in [\k,\theta)\}$ such that 
$\mathfrak a\su \Reg\cap (\mu\sm \theta)$,  
 $|\mathfrak a_\D|=\D$ and  
$(\prod \mathfrak a_
\D, <_{[\mathfrak a_\D]^{<\k}})$ is $\l$-directed. 
\end{itemize}
 \end{claim}
 
 \begin{proof}
$(a)_{\l,\theta,\k}\imply (b)_{\l,\theta,\k}$. 
Choose $G$ witnessing $(a)_{\l,\theta,\k}$.
 We know that $\l$ is regular, and without loss of
generality the vertex set of the graph is $\l$. 
The coloring number is $\ge \theta$ by the 
choice of $G$.  If $H\su G$ has fewer than
 $\l$ vertices then it has coloring number 
$<\theta$  by the minimality of $\l$ . Also the 
complete bipartite graph $K(\k,2^\k)$ 
and even $ K(\kappa , \kappa^+)$ 
is not
 weakly embedded in $G$ because its list-chromatic
 number is $\k^+$ and $\l>2^\k$
and even just $ \lambda > \kappa^+ $. 
Minimality of $\l$ gives more. 
So $(b)_{\l,\theta,\k}$ holds. 

$(b)_{\l,\theta,\k}\imply (c)_{\l,\theta,\k}$.  
See \cite{sh:52} or \cite{sh:266}. 
Assume that the vertex set is $\l$
and let $S=\{\d: (\exists \a\ge \d)(|G[\a]
\cap\d|\ge \theta\}$
where $ G[ \alpha ] =  
\{ \beta : ( \alpha, \beta ) \text{ 
is an edge of }G \} $. 
If $S$ is not stationary then 
using "every subgraph with $<\l$ vertices has
 coloring number $\le\theta$" we conclude 
that $G$ has coloring number $\le \theta$.  
By renaming we get $(c)_{\l,\theta,\k}$.

$(c)_{\l,\theta,\k}\imply (d)_{\l,\theta,\k}$. 
  For each $\D\in [\k,\theta)$ we find, by 
Fodor's lemma, $\a_\D< \lambda $ 
such that 
$\mathcal A_ \partial 
=\{\d\in S:
 |C_\D\cap \a_\D|\ge\D\}$ has cardinality $\l$. 
So $\a_*=\bigcup_\D\a_\D<\l$ satisfies the desired
conclusion for 
the $ \mu $ that is defined as 
 $\mu=|\a_*| $   
so by renaming we are done. 

$(d)_{\l,\theta,\k}\imply (e)_{\l,\theta,\k}$. When, e.g., 
$\partial <\theta\imply \partial^\k<
\lambda $ 
for each $\partial \in [\k,\theta)$ let 
$\lng u_{\partial,\a}:\a<\l\}$ list 
$\mathcal A_ \partial $, 
 and for $\a<\l$ let
 $W_\a=\{\b<\l: |u_{\g,\b}\cap 
u_{\partial,\a}|\ge \k\}$. As
$ | [ u_{\partial , \alpha } ]^\kappa | 
\le \partial < \lambda = \cf( \lambda ) $, 
 the 
set $W_\partial$ is bounded in $\l$, 
 hence for some club $E_\partial\su \l$ 
it holds that $\a<\b\in E_\partial \imply 
 |u_{\g,\a}\cap u_{\partial,\b}|<\k$,
 so $\{u_{\partial,\a}:\a\in E_\partial\}$ is as required. 
 
 $(e)_{\l,\theta,\k}\imply (f)_{\l,\theta,\k}$ if $\partial<\theta\imply \partial^\k<\l$.
By \cite{sh:410} 6.1.  
 
 \end{proof} 
 
%
%
%
%
 
%
%

\end{document}